\documentclass[12pt]{amsart}
\usepackage{amssymb}
\usepackage{mathtools}
\usepackage[english]{babel}
\usepackage{epsfig}
\usepackage{mathptmx}
\usepackage{amsmath,amsfonts,amssymb}
\usepackage{mathtools}
\usepackage{mathrsfs}
\usepackage[all]{xy}
\usepackage{graphicx}
\usepackage{latexsym}
\usepackage{verbatim}
\usepackage{xcolor}

\newcommand{\hide}[1]{}
\numberwithin{equation}{section}

\newcommand{\C}{\mathbb C}

\def\1#1{\overline{#1}}
\def\2#1{\widetilde{#1}}
\def\3#1{\widehat{#1}}
\def\4#1{\mathbb{#1}}
\def\5#1{\frak{#1}}
\def\6#1{{\mathcal{#1}}}

\newcommand{\mcite}[1]{\csname b@#1\endcsname}

\theoremstyle{theorem}

\newtheorem{satz}{Theorem}

\newtheorem{theorem}{Theorem}[section]
\newtheorem{lemma}[theorem]{Lemma}

\theoremstyle{definition}

\theoremstyle{remark}
\newtheorem{remark}{Remark}

\numberwithin{equation}{section}

\title[A Bernstein--type inequality and Carleson's embedding theorem]{A sharp Bernstein--type inequality and application to the\\[2mm] Carleson embedding theorem with matrix weights}
\author[D. Kraus]{Daniela Kraus}
\address{D. Kraus: Department of Mathematics, University of W\"urzburg, Emil Fischer Strasse 40, 97074, W\"urzburg, Germany.} \email{dakraus@mathematik.uni-wuerzburg.de}

\author[A.~Moucha]{Annika Moucha$^\S$}
\address{A. Moucha: Department of Mathematics, University of W\"urzburg, Emil Fischer Strasse 40, 97074, W\"urzburg, Germany.} \email{annika.moucha@mathematik.uni-wuerzburg.de}

\author[O. Roth]{Oliver Roth}
\address{O. Roth: Department of Mathematics, University of W\"urzburg, Emil Fischer Strasse 40, 97074, W\"urzburg, Germany.} \email{roth@mathematik.uni-wuerzburg.de}

\date{\today}

\thanks{$^\S\,$Partially supported by the Alexander von Humboldt Stiftung}

\long\def\REM#1{\relax}

\begin{document}
\selectlanguage{english}
\begin{abstract}
  We prove a sharp Bernstein--type inequality for complex polynomials which are positive and satisfy a polynomial growth condition on the positive real axis. This leads to an improved upper estimate in the recent work of Culiuc and Treil \cite{CuliucTreil}  on  the weighted martingale Carleson embedding theorem with matrix weights. In the scalar case this new upper bound is optimal.

\end{abstract}

\maketitle

\section{Result}

\begin{lemma} \label{lem:1}
  Let $n$ be a positive integer and $p: \C \to \C$  a polynomial such that $p(s) \ge 0$ for all $s \ge 0$ and
\begin{equation} \label{eq:bound}
  |p(s)| \le s^{-1} (1+s)^n\,  \qquad \text{ for all } s>0 \, .
  \end{equation}
  Then
\begin{equation} \label{eq:1}
|p(0)| \le n^2 \, ,
\end{equation}
with equality if
\begin{equation} \label{eq:chebyshev}
  p(s)=p_n(s):=\frac{1}{2} \frac{(s+1)^n}{s} \left(1-T_n\left(\frac{1-s}{1+s}\right)\right) \, .
  \end{equation}
Here, $T_n(x)=\cos(n \arccos x)$ is the $n$-th Chebyshev polynomial of the first kind.
\end{lemma}

The source of motivation for Lemma \ref{lem:1} has been the recent work of Culiuc and Treil \cite{CuliucTreil} on the Carleson embedding theorem with matrix weights. In fact, Lemma 2.2 in \cite{CuliucTreil}, which they attribute to F.~Nazarov and M.~Sodin, provides the  (weaker) estimate 
\begin{equation} \label{eq:bct}
|p(0)| \le e^2 n^2 \, 
\end{equation}
for any polynom $p : \C \to \C$ satisfying (\ref{eq:bound}). Developing a sophisticated Bellman function technique and making use of estimate (\ref{eq:bct}),  Culiuc and Treil \cite{CuliucTreil}  proved  the following result (\cite[Theorem 1.2]{CuliucTreil}). We refer 
 to \cite{CuliucTreil} for the relevant terminology and notation.

 \begin{satz}[Carleson embedding theorem for matrix weights] \label{satz:a}
Let $W$ be a $d \times d$ matrix--valued measure and let $A_I$, $I \in \mathcal{D}$ be a sequence of
   positive semidefinite $d \times d$ matrices. Then the following are equivalent:

 \medskip
   \begin{itemize}
   \item[(i)] $\displaystyle \sum \limits_{I \in \mathcal{D}} \left| \left| A_I^{1/2} \langle( W^{1/2} f \rangle_I \right|\right|^2 \, |I| \le A ||f||^2_{L^2} \, .$\\[2mm]
   \item[(ii)] $\displaystyle \sum \limits_{I \in \mathcal{D}} \left| \left| A_I^{1/2} \langle W f \rangle_I \right|\right|^2 \, |I| \le A ||f||^2_{L^2} \, .$\\[2mm]
   \item[(iii)]  $\displaystyle \frac{1}{|I_0|} \sum \limits_{I \in \mathcal{D}, I \subset I_0} \langle W \rangle_I A_I \langle W\rangle_I \, |I| \le B \langle W \rangle_{I_0} $ for all $I_0 \in \mathcal{D}$.
   \end{itemize}
   Moreover, the best constants $A$ and $B$ satisfy $B \le A \le C B$, where $C=C(d)=4 e^2 d^2$.
   \end{satz}

In fact, the proof of Theorem \ref{satz:a}  in \cite{CuliucTreil} requires the estimate (\ref{eq:bct}) only for polynomials $p : \C \to \C$ with degree $n=2d$, which  satisfy (\ref{eq:bound}) and  are \textit{real and positive on the positive real axis}. Therefore Lemma \ref{lem:1} implies
that one can take
$$ C(d)=4 d^2$$
instead of $C(d)=4 e^2 d^2$ in Theorem \ref{satz:a}. In the scalar case ($d=1$) this new upper bound produces the upper estimate $A \le 4B$, which is known to be optimal \cite[Theorem 3.3]{NTV}.

 \begin{remark}
   The method we use for the proof of Lemma \ref{lem:1} can also be used to improve   the bound (\ref{eq:bct}) given by \cite[Lemma 2.2]{CuliucTreil}, which holds for any polynomial $p: \C \to \C$ satisfying (\ref{eq:bound}). This leads to
   \begin{equation} \label{eq:b1}
    |p(0)| \le 2n^2-n  \, ,
  \end{equation}
see the next section for the proof. The estimate (\ref{eq:b1}) is presumably  not best possible. 
   \end{remark}

  \section{Proofs}

  The idea is to view both estimates, (\ref{eq:1}) and (\ref{eq:b1}), as Bernstein--type estimates. Recall that for a polynomial $h$ of degree $N$ the classical Bernstein inequality says that
  $$ \max \limits_{|z|=1} |h'(z)| \le N \cdot \max \limits_{|z|=1} |h(z)| \, .$$

\begin{proof}[Proof of Lemma \ref{lem:1}]
By assumption, $p: \C \to \C$ is a polynomial satisfying (\ref{eq:bound}) and
$p(s) \ge 0$ for all $s \ge 0$. Then $q(z):=z p(z)$ is polynomial of degree at most $n$ with $q(0)=0$,  $p(0)=q'(0)$, and  $q(s) \ge 0$ for all $s \ge 0$.
We define the auxiliary function
$$ f(z):=\frac{(1+z)^{2n}}{(4z)^n} q\left(- \left(\frac{1-z}{1+z} \right)^2 \right)=\sum \limits_{k=-n}^n a_k z^k \, ,$$
a Laurent polynomial of degree $\le n$. It is not difficult to see that the growth condition (\ref{eq:bound}) for $p$ implies the uniform bound
$$ |f(z)| \le 1 \quad \text{ for all } |z|=1 \, .$$
We also note that
$$ p(0)=q'(0)=-2 f''(1) \, ,$$
so our task is to find the best upper bound for $|f''(1)|$. 

\medskip

In order to find such an estimate, it turns out to be essential that 
the auxiliary function $f$ is  \textit{real and positive (i.e., $\ge 0$)} on $|z|=1$. To see this just note that
$$k(z)=\frac{z}{(1+z)^2}=\frac{1}{4} \left(1-\left( \frac{1-z}{1+z} \right)^2 \right) \, $$
is the Koebe function, familiar from the classical theory of univalent functions, which maps the unit circle $|z|=1$
onto the half--line $[1/4,+\infty)$. Hence, on $|z|=1$, $f(z)$ is the product of two real and positive functions.

\medskip

We are thus in a position to  apply the Fej\'er--Riesz theorem \cite{F} for the Laurent polynomial $f$. This gives us a complex polynomial $P$ of degree $\le n$ with no zeros in $|z|<1$ such that
$$ f(z)=P(z) \overline{P(1/\overline{z})} \, , \qquad z \in \C \setminus \{0\} \, .$$
Clearly, $|P(z)| \le 1$ for all $|z|=1$.  We can therefore  apply a sharpening of Bernstein's inequality due to P.~Lax \cite{Lax} (confirming an earlier conjecture of Erd\"os) which asserts that
$$ \max \limits_{|z|=1} |P'(z)| \le \frac{n}{2} \cdot \max \limits_{|z|=1} |P(z)| \le \frac{n}{2} \, .$$
In particular, $$|p(0)|=|q'(0)|=2 |f''(1)| =4 |P'(1)|^2 \le n^2\, ,$$ proving (\ref{eq:1}).
Clearly, the polynomial $P_n(z)=(z^{n}-1)/2$ has the property $|P'_n(1)|=n/2$, so $|f_n''(1)|=n^2/2$ for $f_n(z):=P_n(z) \overline{P_n(1/\overline{z})}$. It is easy to see that
$$ f_n(z)=\frac{(1+z)^{2n}}{(4z)^n} q_n\left(- \left(\frac{1-z}{1+z} \right)^2 \right)$$
for a polynomial $q_n$ of degree at most $n$ with $q_n(0)=0$, and  it is straightforward to check that $p_n(z):=q_n(z)/z$ has the form (\ref{eq:chebyshev}).
\end{proof}

\begin{proof}[Proof of (\ref{eq:b1})]
By assumption,  $p: \C \to \C$ is a polynomial satisfying (\ref{eq:bound}).
Then $q(z):=z p(z)$ is polynomial of degree at most $n$ with $q(0)=0$ and $p(0)=q'(0)$.
We define, closely following the proof of \cite[Lemma 2.2]{CuliucTreil},  the auxiliary function
$$ g(z):=\frac{(1+z)^{2n}}{4^n} q\left(- \left(\frac{1-z}{1+z} \right)^2 \right) \, ,$$
a polynomial of degree $N\le 2 n$.
As before,  the polynomial $g$  has the property that
$$ |g(z)| \le 1 \quad \text{ for all } |z|=1 \, .$$
Now note that
$$ p(0)=-2 g''(1) \, .$$
Hence, we could  apply the classical Bernstein inequality twice, first for $g'$ and then for $g''$, but this would result
in $$|p(0)|=2 |g''(1)| \le 2 N(N-1) \le 4n (2n-1) \, ,$$ which is not particularly good. 
However, as observed in \cite[Proof of Lemma 2.2]{CuliucTreil} we can assume without loss of generality that $g$ has no zeros in $|z|<1$. We can therefore  apply as above the inequality of Lax which leads to
$$ \max \limits_{|z|=1} |g'(z)| \le \frac{N}{2} \cdot \max \limits_{|z|=1} |g(z)| \le n \, .$$
This brings us in a position to apply  Corollary 14.2.8 in \cite{RS} for the polynomial $g'$ which has  degree $\le 2n-1$. Hence
$$ |g''(z)|+|(2n-1) g'(z)-z g''(z)| \le n (2n-1) \, , \qquad |z| \le 1 \, .$$
Taking  $z=1$ and noting that $g'(1)=n q(0)=0$, gives $2 |g''(1)| \le n (2n-1)$, as required.
\end{proof}

\section{Remark}

The polynomials $p$ which occur in the proof of Theorem \ref{satz:a} in \cite{CuliucTreil} are of the form
$$ p(s)=\sum \limits_{I \in \mathcal{D}} p_I(s) |I| \, ,$$
with  $p_I(s) \ge 0$ for all $s \ge 0$ and  each $p_I$  a polynomial of degree at most $2(d-1)$.
The extremal polynomial $p_{2d}$ in Lemma \ref{lem:1} has degree $2(d-1)$ and \textit{all its $2(d-1)$ zeros are on the positive real axis and are double zeros}. This implies that
$$ p(s)=p_{2d}(s) \quad \Longleftrightarrow \quad \forall_{I \in \mathcal{D}} \exists_{c(I) \ge 0} \, p_I|I|=c(I) p_{2d} \, .$$
Hence the extremal polynomial $p_{2d}$ of Lemma \ref{lem:1} shows up in the proof of Theorem \ref{satz:a} only if each $p_I$ is a multiple of $p_{2d}$. 

\medskip

It is not clear whether there exist extremal polynomials in Lemma \ref{lem:1} other than $p_n$.

\section{Acknowledgement} The authors would like to  thank Stefanie Petermichl for various helpful discussions and remarks.

\end{document}